\documentclass[preprint,9pt]{elsarticle}

\usepackage{amssymb}
\usepackage{amsmath}
\usepackage{amsthm}
\usepackage{enumerate}
\newtheorem{thrm}{Theorem}[section]
\newtheorem{lem}[thrm]{Lemma}

\newtheorem{exam}[thrm]{Example}
\newtheorem{cor}[thrm]{Corollary}

\theoremstyle{definition}
\newtheorem{definition}[thrm]{Definition}

\newtheorem{conclusion}[thrm]{Conclusion}

\journal{}

\begin{document}

\begin{frontmatter}


\cortext[cor1]{Corresponding author (+903562521616-3087)}

\title{$k$ sequences of Generalized Van der Laan and Generalized Perrin Polynomials}


\author[rvt]{Kenan Kaygisiz\corref{cor1}}
\ead{kenan.kaygisiz@gop.edu.tr}
\author[rvt]{Adem \c{S}ahin}
\ead{adem.sahin@gop.edu.tr}
\address[rvt]{Department of Mathematics, Faculty of Arts and Sciences,
Gaziosmanpa\c{s}a University, 60250 Tokat, Turkey}

\begin{abstract}
In this paper, we present $k$ sequences of Generalized Van der Laan
Polynomials and Generalized Perrin Polynomials using Genaralized
Fibonacci and Lucas Polynomials. We give some properties of these
polynomials. We also obtain generalized order-$k$ Van der Laan
Numbers, $k$
sequences of generalized order-$k$ Van der Laan Numbers, generalized order-$%
k $ Perrin Numbers and $k$ sequences of generalized order-$k$ Perrin
Numbers. In addition, we examine the relationship between them.
\end{abstract}

\end{frontmatter}


\section{Introduction}
Fibonacci, Lucas, Pell and Perrin numbers are known for a long time.
There are a lot of studies, relations, and applications of them.
Generalization of this numbers has been studied by many researchers.

Miles [11] defined generalized order-$k$ Fibonacci numbers(GO$k$F)
in 1960.
Er [1] defined $k$ sequences of generalized order-$k$ Fibonacci Numbers($k$SO$k$%
F) and gave matrix representation for this sequences in 1984. Kalman
[2] obtained a Binet formula for these sequences in 1982. Karaduman
[3], Ta\c{s}\c{c}\i\ and K\i l\i \c{c} [13] studied on these
sequences. K\i l\i \c{c} and Ta\c{s}\c{c}\i\ [7] defined $k$
sequences of generalized order-$k$ Pell Numbers($k$SO$k$P) and
obtained sums properties by using matrix method. Kaygisiz and
Bozkurt [5] studied on generalization of Perrin numbers. Yýlmaz and
Bozkurt [14] give some properties of Perrin and Pell numbers.

Meanwhile, MacHenry [8] defined generalized Fibonacci polynomials
$(F_{k,n}(t))$, Lucas polynomials $(G_{k,n}(t))$ in 1999, studied on
these polynomials in [9] and defined matrices $A_{(k)}^{\infty } $
and $D_{(k)}^{\infty }$ in [10]. These studies of MacHenry include
most of other studies mentioned above. For example, $A_{(k)}^{\infty
}$ is reduced to $k$ sequences of generalized order-$k$ Fibonacci
Numbers and $A_{(k)}^{\infty }$
is reduced to $k$ sequences of generalized order-$k$ Pell Numbers when $t_{1}=2$ and $%
t_{i}=1$ (for $2\leq i\leq k$). In addition Binet formulas for
($k$SO$k$F) and ($k$SO$k$P) can be obtained by using equation
(\ref{adem}). This analogy shows the importance of the matrices $A_{(k)}^{\infty }$ and $%
D_{(k)}^{\infty }$ and Generalized Fibonacci and Lucas polynomials.
Based on this idea Kaygisiz and \c{S}ahin defined $k$
sequences of generalized order-$k$ Lucas Numbers using $G_{k,n}(t)$ and $%
D_{(k)}^{\infty }$ in[4]$.$

\bigskip In this article, we first present $k$ sequences of Generalized Van
der Laan and Perrin Polynomials($V_{k,n}^{i}(t)$ and
$R_{k,n}^{i}(t)$) using
Genaralized Fibonacci and Lucas Polynomials and obtain generalized order-$%
k$ Van der Laan and Perrin Numbers, $k$ sequences of generalized
order-$k$ Van der Laan and Perrin Numbers by the help of these
polynomials and matrices $A_{(k)}^{\infty }$ and $D_{(k)}^{\infty
}$. In addition, we examine the relationship between them and
explore some of the properties of these sequences. We believe that
our result are important, especially, for those who are interested
in well known Fibonacci, Lucas, Pell and Perrin sequences and their
generalization.

\bigskip

MacHenry [8] defined generalized Fibonacci polynomials
$(F_{k,n}(t))$, Lucas polynomials $(G_{k,n}(t))$ and obtained
important relations between generalized Fibonacci and Lucas
polynomials in [9], where $t_{i}$ $(1\leq i\leq k)$
are constant coefficients of the core polynomial%
\begin{equation}
P(x;t_{1},t_{2},\ldots ,t_{k})=x^{k}-t_{1}x^{k-1}-\cdots -t_{k},
\label{core}
\end{equation}%
which is denoted by the vector
\begin{equation}
t=(t_{1},t_{2},\ldots ,t_{k}). \label{vkt}
\end{equation}%
$F_{k,n}(t)$ is defined inductively by
\begin{eqnarray*}
F_{k,n}(t) &=&0,\text{ }n<0 \\
F_{k,0}(t) &=&1 \\
F_{k,1}(t) &=&t_{1} \\
F_{k,n+1}(t) &=&t_{1}F_{k,n}(t)+\cdots +t_{k}F_{k,n-k+1}(t).
\end{eqnarray*}%

$G_{k,n}(t_{1},t_{2},\ldots ,t_{k})$ is defined by%
\begin{eqnarray*}
G_{k,n}(t) &=&0,\text{ }n<0 \\
G_{k,0}(t) &=&\text{ }k \\
G_{k,1}(t) &=&t_{1}\text{ } \\
G_{k,n+1}(t) &=&t_{1}G_{k,n}(t)+\cdots +t_{k}G_{k,n-k+1}(t).
\end{eqnarray*}

In [10], matrices $A_{(k)}^{\infty }$ and $D_{(k)}^{\infty }$ are
defined by using the following matrix,

\begin{equation}
A_{(k)}=\left[
\begin{array}{ccccc}
0 & 1 & 0 & \ldots & 0 \\
0 & 0 & 1 & \ldots & 0 \\
\vdots & \vdots & \vdots & \ddots & \vdots \\
0 & 0 & 0 & \ldots & 1 \\
t_{k} & t_{k-1} & t_{k-2} & \ldots & t_{1}%
\end{array}%
\right] \text{ .}  \label{ak}
\end{equation}

$A_{(k)}^{\infty }$ is obtained by multiplying $A_{(k)}$ and
$A_{(k)}^{-1}$ by the vector $t$ in (\ref{vkt}). For $k=3,$ $A_{(k)}^{\infty }$ looks like this%
\begin{equation*}
A_{(3)}^{\infty }=\left[
\begin{array}{ccc}
\cdots & \cdots & \cdots \\
S_{(-n,1^{2})} & -S_{(-n,1)} & S_{(-n)} \\
\cdots & \cdots & \cdots \\
S_{(-3,1^{2})} & -S_{(-3,1)} & S_{(-3)} \\
1 & 0 & 0 \\
0 & 1 & 0 \\
0 & 0 & 1 \\
t_{3} & t_{2} & t_{1} \\
\cdots & \cdots & \cdots \\
S_{(n-1,1^{2})} & -S_{(n-1,1)} & S_{(n-1)} \\
S_{(n,1^{2})} & -S_{(n,1)} & S_{(n)} \\
\cdots & \cdots & \cdots%
\end{array}%
\right] \bigskip \ \ \ \
\end{equation*}%
and

\ \
\begin{equation}
A_{(k)}^{n}=\left[
\begin{array}{ccccc}
(-1)^{k-1}S_{(n-k+1,1^{k-1})} & \cdots &
(-1)^{k-j}S_{(n-k+1,1^{k-j})} &
\cdots & S_{(n-k+1)} \\
\cdots & \cdots & \cdots & \cdots & \cdots \\
(-1)^{k-1}S_{(n,1^{k-1})} & \cdots & (-1)^{k-j}S_{(n,1^{k-j})} &
\cdots &
S_{(n)}%
\end{array}%
\right] \ \ \   \label{gm}
\end{equation}%
where%
\begin{equation}
\ \ S_{(n-r,1^{r})}=(-1)^{r}\sum\limits_{j=r+1}^{n}t_{j}S_{(n-j)},\
0\leq r\leq n.
\end{equation}

Derivative of the core polynomial (\ref{core}) is
\begin{equation*}
P%
{\acute{}}%
(x)=kx^{k-1}-t_{1}(k-1)x^{k-2}-\cdots -t_{k-1},
\end{equation*}%
which is represented by the vector%
\begin{equation}
(-t_{k-1},\ldots ,-t_{1}(k-1),k). \label{dkv}
\end{equation}%
Multiplying $A_{(k)}$ and $A_{(k)}^{-1}$ by the vector (\ref{dkv})
gives the matrix $D_{(k)}^{\infty }.$

Right hand column of $A_{(k)}^{\infty }$ contains sequence of the
generalized Fibonacci polynomials $F_{k,n}(t)$. In addition, the right hand column of $%
D_{(k)}^{\infty }$ contains sequence of the generalized Lucas polynomials $%
G_{k,n}(t)$.

\bigskip

For easier reference, we have state some theorems which will be used
in the subsequent section. We also give some sequences mentioned
above.

\begin{thrm}
$[$9$]$ $F_{k,n}(t)$ and $G_{k,n}(t)$ are Generalized Fibonacci and
Lucas
polynomials respectively,%
\begin{equation*}
\sum\limits_{j=1}^{k}\frac{\partial G_{k,n}(t)}{\partial t_{j}}%
t_{j}=n\bigskip F_{k,n+1}(t)
\end{equation*}%
and%
\begin{equation*}
n\bigskip
F_{k,n+1}(t)=\sum\limits_{r=1}^{k}G_{k,r}(t)F_{k,n-r+1}(t).
\end{equation*}
\ \ \ \ \ \ \ \ \ \ \ \ \ \ \ \ \ \ \ \ \ \ \ \ \ \ \ \ \ \ \ \ \ \
\ \ \ \ \ \ \ \ \ \ \ \ \ \ \ \ \ \ \ \ \ \ \
\end{thrm}

\bigskip

\begin{thrm}
$[$9$]$ Let $\lambda _{j}$ be the roots of the polynomial (\ref%
{core}) and let%
\begin{equation*}
\Delta _{k}=\left\vert
\begin{array}{ccc}
1 & \cdots & 1 \\
\lambda _{1} & \cdots & \lambda _{k} \\
\vdots &  & \vdots \\
\lambda _{1}^{k-1} & \cdots & \lambda _{k}^{k-1}%
\end{array}%
\right\vert \text{ and }\Delta _{k,n}=\left\vert
\begin{array}{ccc}
1 & \cdots & 1 \\
\lambda _{1} & \cdots & \lambda _{k} \\
\vdots &  & \vdots \\
\lambda _{1}^{k-2} & \cdots & \lambda _{k}^{k-2} \\
\lambda _{1}^{n+k-2} & \cdots & \lambda _{k}^{n+k-2}%
\end{array}%
\right\vert ,
\end{equation*}%
then we have%
\begin{equation}
F_{k,n+1}(t)=\frac{\Delta _{k,n}}{\Delta _{k}}.  \label{adem}
\end{equation}
\end{thrm}

\bigskip

\begin{thrm}
$[$10$]$$A_{(k)}$ is $k\times k$ matrix in (\ref{ak}) then,
\begin{equation*}
\det A_{(k)}=(-1)^{k+1}t_{k}
\end{equation*}%
and%
\begin{equation}
\det A_{(k)}^{n}=(-1)^{n(k+1)}t_{k}^{n}.  \label{a1}
\end{equation}
\end{thrm}

Miles [11] defined generalized order-$k$ Fibonacci numbers(GO$k$F) as,%
\begin{equation}
f_{k,n}=\sum\limits_{j=1}^{k}f_{k,n-j}\
\end{equation}%
for $n>k\geq 2$, with boundary conditions:
$f_{k,1}=f_{k,2}=f_{k,3}=\cdots =f_{k,k-2}=0$ and
$f_{k,k-1}=f_{k,k}=1.$\newline

Er [1] defined $k$ sequences of generalized order-$k$ Fibonacci Numbers($k$SO%
$k$F) as; for $n>0,$ $1\leq i\leq k$%
\begin{equation}
f_{k,n}^{\text{ }i}=\sum\limits_{j=1}^{k}c_{j}f_{k,n-j}^{\text{ }i}\
\
\end{equation}%
with boundary conditions for $1-k\leq n\leq 0,$

\begin{equation*}
f_{k,n}^{\text{ }i}=\left\{
\begin{array}{c}
1\text{ \ \ \ \ \ if \ }i=1-n, \\
0\text{ \ \ \ \ \ \ \ \ \ otherwise,}%
\end{array}%
\right.
\end{equation*}%
where $c_{j}$ $(1\leq j\leq k)$ are constant coefficients, $f_{k,n}^{\text{ }%
i}$ is the $n$-th term of $i$-th sequence of order $k$ generalization. $k$%
-th column of this generalization involves the Miles generalization for $%
i=k, $ i.e. $f_{k,n}^{k}=f_{k,k+n-2}.$

\bigskip K\i l\i \c{c} [7] defined $k$ sequences of generalized order-$k$
Pell Numbers($k$SO$k$P) as; for $n>0,$ $1\leq i\leq k$%
\begin{equation}
P_{k,n}^{\text{ }i}=2P_{k,n-1}^{\text{ }i}\ +P_{k,n-2}^{\text{
}i}+\cdots +\ P_{k,n-k}^{\text{ }i}
\end{equation}%
with boundary conditions for $1-k\leq n\leq 0,$

\begin{equation*}
P_{k,n}^{\text{ }i}=\left\{
\begin{array}{c}
1\text{ \ \ \ \ \ if \ }n=1-i, \\
0\text{ \ \ \ \ \ \ \ \ \ otherwise}%
\end{array}%
\right.
\end{equation*}%
where $P_{k,n}^{\text{ }i}$ is the $n$-th term of $i$-th sequence of order $%
k $ generalization.

\bigskip

The well-known Cordonnier(Padovan) sequence $\left\{ C_{n}\right\} $
is defined recursively by the equation,
\begin{equation*}
C_{n}=C_{n-2}+C_{n-3},\ \text{for }n>3\ \
\end{equation*}%
where $P_{1}=1,$ $P_{2}=1,$ $P_{3}=1.$

Van der Laan sequence $\left\{ V_{n}\right\} $ is defined
recursively by the
equation,%
\begin{equation*}
V_{n}=V_{n-2}+V_{n-3},\text{\ for }n>3
\end{equation*}%
where $V_{1}=1,$ $V_{2}=0,$ $V_{3}=1$.

Perrin sequence $\left\{ R_{n}\right\} $ is defined recursively by
the
equation,%
\begin{equation*}
R_{n}=R_{n-2}+R_{n-3},\text{\ for }n>3
\end{equation*}%
where $R_{1}=0,$ $R_{2}=2,$ $R_{3}=3 [12].$

\bigskip

In this paper we studied on generalized order-$k$ Van der Laan numbers $%
v_{k,n}$ and $k$ sequences of the generalized order-$k$ Van der Laan
numbers $v_{k,n}^{i}$ with the help of $k$ sequences of Generalized
Van der Laan Polynomials and generalized order-$k$ Perrin numbers
$r_{k,n}$ and $k$ sequences of the generalized order-$k$ Perrin
numbers $r_{k,n}^{i}$ with the help of $k$ sequences of Generalized
Perrin Polynomials$.$

\section{Generalized Van der Laan and Perrin Polynomials}

\bigskip

Firstly we define generalized Van der Laan polynomial and $k$
sequences of generalized Van der Laan polynomial by the help of
generalized Fibonacci polynomials $(F_{k,n}(t))$ and matrices
$A_{(k)}^{\infty }.$

\bigskip

\begin{definition}
Generalized Fibonacci polynomials $(F_{k,n}(t))$ is called
generalized Van der Laan polynomials in case $t_{1}=0$ for $k\geq 3$
. So generalized Van der Laan polynomials are
\begin{eqnarray*}
V_{k,0}(t) &=&0 \\
V_{k,1}(t) &=&1 \\
V_{k,2}(t) &=&0 \\
V_{k,3}(t) &=&t_{2}V_{k,1}(t) \\
V_{k,4}(t) &=&t_{2}V_{k,2}(t)+t_{3}V_{k,1}(t) \\
&&\vdots \\
V_{k,k-1}(t) &=&t_{2}V_{k,k-3}(t)+\cdots +t_{k-1}V_{k,1}(t)
\end{eqnarray*}
and for $n\geq k$%
\begin{equation*}
V_{k,n}(t)=\sum\limits_{i=2}^{k}t_{i}V_{k,n-i}(t).
\end{equation*}
\end{definition}
\bigskip

For $k\geq 3$ substituting $t_{1}=0$, generalized Fibonacci polynomials $%
(F_{k,n}(t))$ and matrices $A_{(k)}^{\infty }$ are reduced to
following polynomials; for $n>0,$ $1\leq i\leq k$%
\begin{equation}
V_{k,n}^{i}(t)=\sum\limits_{j=2}^{k}t_{j}V_{k,n-j}^{i}(t)
\end{equation}%
with boundary conditions for $1-k\leq n\leq 0,$

\begin{equation*}
V_{k,n}^{i}(t)=\left\{
\begin{array}{c}
1\text{ \ \ \ \ \ if \ }k=i-n, \\
0\text{ \ \ \ \ \ \ \ \ \ otherwise,}%
\end{array}%
\right.
\end{equation*}%
where $V_{k,n}^{i}(t)$ is the $n$-th term of $i$-th sequence of
order $k$ generalization.

\bigskip

\begin{definition}
\bigskip The polynomials derived in (12) is called $k$ sequences of
generalized Van der Laan polynomials.
\end{definition}

\bigskip We note that for $i=k$ and $n\geqslant0$, $V_{k,n-1}^{i}(t)=V_{k,n}(t).$

\begin{exam}
\bigskip We give $k$ sequences of generalized Van der Laan polynomial $%
V_{k,n}^{i}(t)$ for $k=3$ and $k=4$%
\begin{equation*}
\begin{tabular}{|c|c|c|c|}
\hline $n\setminus i$ & $1$ & $2$ & $3$ \\ \hline $-2$ & $1$ & $0$ &
$0$ \\ \hline $-1$ & $0$ & $1$ & $0$ \\ \hline $0$ & $0$ & $0$ & $1$
\\ \hline $1$ & $t_{3}$ & $t_{2}$ & $0$ \\ \hline $2$ & $0$ &
$t_{3}$ & $t_{2}$ \\ \hline $3$ & $t_{2}t_{3}$ & $t_{2}^{2}$ &
$t_{3}$ \\ \hline $4$ & $t_{3}^{2}$ & $2t_{2}t_{3}$ & $t_{2}^{2}$ \\
\hline $5$ & $t_{2}^{2}t_{3}$ & $t_{2}^{3}+t_{3}^{2}$ &
$2t_{2}t_{3}$ \\ \hline $6$ & $2t_{2}t_{3}^{2}$ & $3t_{2}^{2}t_{3}$
& $t_{2}^{3}+t_{3}^{2}$ \\ \hline
$7$ & $t_{3}^{3}+t_{2}^{3}t_{3}$ & $t_{2}^{4}+3t_{2}t_{3}^{2}$ & $%
3t_{2}^{2}t_{3}$ \\
$8$ & $\vdots $ & $\vdots $ & $\vdots $%
\end{tabular}%
\text{ \ }%
\begin{tabular}{|c|c|c|c|c|}
\hline $n\setminus i$ & $1$ & $2$ & $3$ & $4$ \\ \hline $-3$ & $1$ &
$0$ & $0$ & $0$ \\ \hline $-2$ & $0$ & $1$ & $0$ & $0$ \\ \hline
$-1$ & $0$ & $0$ & $1$ & $0$ \\ \hline $0$ & $0$ & $0$ & $0$ & $1$
\\ \hline $1$ & $t_{4}$ & $t_{3}$ & $t_{2}$ & $0$ \\ \hline $2$ &
$0$ & $t_{4}$ & $t_{3}$ & $t_{2}$ \\ \hline $3$ & $t_{2}t_{4}$ &
$t_{2}t_{3}$ & $t_{4}+t_{2}^{2}$ & $t_{3}$ \\ \hline
$4$ & $t_{3}t_{4}$ & $t_{2}t_{4}+t_{3}^{2}$ & $2t_{2}t_{3}$ & $%
t_{4}+t_{2}^{2}$ \\ \hline
$5$ & $\vdots $ & $\vdots $ & $\vdots $ & $\vdots $%
\end{tabular}%
\end{equation*}
\end{exam}

\bigskip In addition%
\begin{equation*}
V_{(k)}=\left[
\begin{array}{ccccc}
0 & 1 & 0 & \ldots & 0 \\
0 & 0 & 1 & \ldots & 0 \\
\vdots & \vdots & \vdots & \ddots & \vdots \\
0 & 0 & 0 & \ldots & 1 \\
t_{k} & t_{k-1} & t_{k-2} & \ldots & 0%
\end{array}%
\right] \text{ }
\end{equation*}%
is the generator matrix of $k$ sequences of generalized Van der Laan
polynomials. Matrix $V_{(k)}^{\infty }$ is obtained by multiplying
$V_{(k)}$ and $V_{(k)}^{-1}$ by the vector $v=(t_{k}, t_{k-1},
t_{k-2} \ldots ,0$).

Note that it is also possible to obtain matrix $V_{(k)}^{\infty }$
from matrix $A_{(k)}^{\infty }$ by substituting $t_{1}=0$.

Let $\widetilde{V_{n}}$ be generalized Van der Laan matrix which is
obtained by $n$-th power of $V_{(k)}$ as;
\begin{equation}
\widetilde{V_{n}}=(V_{(k)})^{n}=\left[
\begin{array}{cccc}
V_{k,n-k+1}^{1}(t) & V_{k,n-k+1}^{2}(t) & \ldots & V_{k,n-k+1}^{k}(t) \\
\vdots & \vdots & \ldots & \vdots \\
V_{k,n-1}^{1}(t) & V_{k,n-1}^{2}(t) & \ddots & V_{k,n-1}^{k}(t) \\
V_{k,n}^{1}(t) & V_{k,n}^{2}(t) & \ldots & V_{k,n}^{k}(t)%
\end{array}%
\right] .  \label{e2.2}
\end{equation}%
It is obvious that $V_{(k)}=\widetilde{V_{1}}.$

\bigskip

\begin{cor}
\bigskip Let $\widetilde{V_{n}}$ be as in (\ref{e2.2}). Then
\begin{equation*}
\det \widetilde{V_{n}}=(-1)^{n(k+1)}t_{k}^{n}.
\end{equation*}
\end{cor}

\begin{proof}
\bigskip It is direct from Theorem 1.3.
\end{proof}

\bigskip We define generalized Perrin polynomial and matrix $R_{(k)}^{\infty
}$ by the help of generalized Lucas polynomials $(G_{k,n}(t))$ and matrices $%
D_{(k)}^{\infty }.$

\bigskip

\begin{definition}
Generalized Lucas polynomials $(G_{k,n}(t))$ are called generalized
Perrin polynomials in case $t_{1}=0$ for $k\geq 3$ . So generalized
Perrin polynomials are;
\begin{eqnarray*}
R_{k,0}(t) &=&k \\
R_{k,1}(t) &=&0 \\
R_{k,2}(t) &=&2t_{2} \\
R_{k,3}(t) &=&t_{2}R_{k,1}(t)+3t_{3} \\
R_{k,4}(t) &=&t_{2}R_{k,2}(t)+t_{3}R_{k,1}(t)+4t_{4} \\
&&\vdots \\
R_{k,k-1}(t) &=&t_{2}R_{k,k-3}(t)+\cdots +t_{k-1}R_{k,1}(t)+kt_{k}
\end{eqnarray*}%
and for $n\geq k$,%
\begin{equation*}
R_{k,n}(t)=\sum\limits_{i=2}^{k}t_{i}R_{k,n-i}(t).
\end{equation*}
\end{definition}

\bigskip

\bigskip We obtain matrix $R_{(k)}$ by using row vector%
\begin{equation*}
(-t_{k-1},\ldots ,-t_{2}(k-2),0,k)
\end{equation*}%
which is obtained from coefficient of derivative of core polynomial
(1). $k$-th
row of matrix $R_{(k)}$ is the vector%
\begin{equation*}
(-t_{k-1},\ldots ,-t_{2}(k-2),0,k)
\end{equation*}%
and $i$-th row is
\begin{equation*}
(-t_{k-1},\ldots ,-t_{2}(k-2),0,k)(V_{(k)})^{-(k-i)}
\end{equation*}%
for $1\leq i\leq k-1.$ It looks like%
\begin{equation}
R_{(k)}=\left[
\begin{array}{c}
(-t_{k-1},\ldots ,-t_{2}(k-2),0,k).(V_{(k)})^{-(k-1)} \\
(-t_{k-1},\ldots ,-t_{2}(k-2),0,k).(V_{(k)})^{-(k-2)} \\
\vdots  \\
(-t_{k-1},\ldots ,-t_{2}(k-2),0,k).(V_{(k)})^{-1} \\
(-t_{k-1},\ldots ,-t_{2}(k-2),0,k)%
\end{array}%
\right] _{k\times k}.  \label{rk}
\end{equation}

\bigskip

\bigskip For $k\geq 3$ substituting $t_{1}=0$, generalized Lucas polynomials
$(G_{k,n}(t))$ and matrices $D_{(k)}^{\infty }$ are reduced to
following polynomials $R_{k,n}^{i}(t)$; for $n>0,$ $1\leq i\leq k$%
\begin{equation}
R_{k,n}^{i}(t)=\sum\limits_{j=2}^{k}t_{j}R_{k,n-j}^{i}(t)
\end{equation}%
with boundary conditions for $1-k\leq n\leq 0,$%
\begin{equation*}
R_{(k)}=[a_{k+n,i}]=R_{k,n}^{i}(t).
\end{equation*}

\bigskip

\begin{definition}
\bigskip \bigskip The polynomials $R_{k,n}^{i}(t)$ derived (15) is called $%
k $ sequences of generalized Perrin polynomials.
\end{definition}

For $k\geq 3$ and substituting $t_{1}=0$ matrix $D_{(k)}^{\infty }$
reduces to matrix $R_{(k)}^{\infty }$. Right hand column of
$R_{(k)}^{\infty }$ contains Generalized Perrin polynomials
$R_{k,n}(t)$. $i$-th column of matrix $R_{(k)}^{\infty }$ becomes
$i$-th sequence of $k$ sequences of generalized Perrin polynomials
vice versa.\bigskip\

We note that matrix $R_{(k)}^{\infty }$ for $t_{1}=0$ in matrix $%
D_{(k)}^{\infty }$ and $R_{(k)}^{\infty }$ contains $k$ sequence of
Perrin polynomials $R_{k,n}^{i}(t).$ Let $\widetilde{R_{n}}$ be
generalized Perrin matrix is obtained by $R_{(k)}.(V_{(k)})^{n}$ as;
\begin{equation}
\widetilde{R_{n}}=R_{(k)}.(V_{(k)})^{n}=\left[
\begin{array}{cccc}
R_{k,n-k+1}^{1}(t) & R_{k,n-k+1}^{2}(t) & \ldots & R_{k,n-k+1}^{k}(t) \\
\vdots & \vdots & \ldots & \vdots \\
R_{k,n-1}^{1}(t) & R_{k,n-1}^{2}(t) & \ddots & R_{k,n-1}^{k}(t) \\
R_{k,n}^{1}(t) & R_{k,n}^{2}(t) & \ldots & R_{k,n}^{k}(t)%
\end{array}%
\right] .  \label{e2.4}
\end{equation}

\begin{exam}
\bigskip We give matrix $R_{(k)}^{\infty }$ and $k$ sequences of generalized
Perrin polynomials for $k=3$ respectively,%
\begin{equation*}
R_{(3)}^{\infty }=\left[
\begin{array}{ccc}
-\frac{t_{2}^{3}}{t_{3}^{2}}+3 & -\frac{t_{2}}{t_{3}} & \frac{t_{2}^{2}}{%
t_{3}^{2}} \\
\frac{t_{2}^{2}}{t_{3}} & 3 & -\frac{t_{2}}{t_{3}} \\
-t_{2} & 0 & 3 \\
3t_{3} & 2t_{2} & 0 \\
0 & 3t_{3} & 2t_{2} \\
2t_{2}t_{3} & 2t_{2}^{2} & 3t_{3} \\
\vdots & \vdots & \vdots%
\end{array}%
\right] \text{ \ and } \ %
\begin{tabular}{|c|c|c|c|}
\hline $n\setminus i$ & $1$ & $2$ & $3$ \\ \hline
$-2$ & $-\frac{t_{2}^{3}}{t_{3}^{2}}+3$ & $-\frac{t_{2}}{t_{3}}$ & $\frac{%
t_{2}^{2}}{t_{3}^{2}}$ \\ \hline $-1$ & $\frac{t_{2}^{2}}{t_{3}}$ &
$3$ & $-\frac{t_{2}}{t_{3}}$ \\ \hline $0$ & $-t_{2}$ & $0$ & $3$ \\
\hline $1$ & $3t_{3}$ & $2t_{2}$ & $0$ \\ \hline $2$ & $0$ &
$3t_{3}$ & $2t_{2}$ \\ \hline
$3$ & $2t_{2}t_{3}$ & $2t_{2}^{2}$ & $3t_{3}$ \\
$ \vdots $ & $\vdots $ & $\vdots $ & $\vdots $%
\end{tabular}%
\end{equation*}
\end{exam}

\bigskip

Now we give four Corollary by using properties of Generalized
Fibonacci and Lucas Numbers.

\begin{cor}
$tr(V_{(k)}^{n})=R_{k,n}(t)$ for $n\in
\mathbb{Z}
,$ where $R_{k,n}(t)$ is the Generalized Perrin polynomials.
\end{cor}

\bigskip

\begin{cor}
$V_{k,n}^{i}(t)$ be the $k$ sequences of generalized Perrin and Van
der Laan
polynomials respectively, then%
\begin{equation*}
V_{k,n}^{k}(t)=V_{k,n-1}^{k-1}(t)
\end{equation*}%
\bigskip\ and%
\begin{equation*}
V_{k,n}^{1}(t)=t_{k}V_{k,n-1}^{k}(t)
\end{equation*}
for $n\geq 1.$
\end{cor}

\bigskip

\begin{cor}
\bigskip $R_{k,n}^{i}(t)$ be the $k$ sequences of generalized Perrin
polynomials respectively, then%
\begin{equation*}
R_{k,n}^{k}(t)=R_{k,n-1}^{k-1}(t)
\end{equation*}%
and%
\begin{equation*}
R_{k,n}^{1}(t)=t_{k}R_{k,n-1}^{k}(t)
\end{equation*}
for $n\geq 1.$
\end{cor}

\bigskip

\begin{cor}
$R_{k,n}^{i}(t)$ and $V_{k,n}^{i}(t)$ are $k$ sequences of
generalized
Perrin and Van der Laan polynomials respectively, then%
\begin{equation*}
\sum\limits_{j=1}^{k}\frac{\partial R_{k,n}^{k}(t)}{\partial t_{j}}%
t_{j}=n\bigskip V_{k,n}^{k}(t).
\end{equation*}%
\bigskip
\end{cor}

\bigskip

\begin{thrm}
$R_{k,n}^{i}(t)$ and $V_{k,n}^{i}(t)$ are $k$ sequences of
generalized Perrin and Van der Laan polynomials respectively, then
\begin{equation*}
R_{k,n}^{i}(t)=(-t_{k-1})V_{k,n-k+1}^{i}(t)+\ldots
+(-t_{2}(k-2))V_{k,n-2}^{i}(t)+kV_{k,n}^{i}(t).
\end{equation*}
\end{thrm}

\bigskip

\begin{proof}
Using \bigskip (\ref{e2.2}) and (\ref{e2.4}) we obtain
\begin{eqnarray*}
\widetilde{R_{n}} &=&R_{(k)}\widetilde{V_{n}} \\
\Rightarrow  &&\left[
\begin{array}{cccc}
R_{k,n-k+1}^{1}(t) & R_{k,n-k+1}^{2}(t) & \ldots  & R_{k,n-k+1}^{k}(t) \\
\vdots  & \vdots  & \ldots  & \vdots  \\
R_{k,n-1}^{1}(t) & R_{k,n-1}^{2}(t) & \ddots  & R_{k,n-1}^{k}(t) \\
R_{k,n}^{1}(t) & R_{k,n}^{2}(t) & \ldots  & R_{k,n}^{k}(t)%
\end{array}%
\right]  \\
&=&\left[
\begin{array}{c}
(-t_{k-1},\ldots ,-t_{2}(k-2),0,k).(V_{(k)})^{-(k)} \\
(-t_{k-1},\ldots ,-t_{2}(k-2),0,k).(V_{(k)})^{-(k-1)} \\
\vdots  \\
(-t_{k-1},\ldots ,-t_{2}(k-2),0,k).(V_{(k)})^{-1} \\
(-t_{k-1},\ldots ,-t_{2}(k-2),0,k)%
\end{array}%
\right]  \\
&&\left[
\begin{array}{cccc}
V_{k,n-k+1}^{1}(t) & V_{k,n-k+1}^{2}(t) & \ldots  & V_{k,n-k+1}^{k}(t) \\
\vdots  & \vdots  & \ldots  & \vdots  \\
V_{k,n-1}^{1}(t) & V_{k,n-1}^{2}(t) & \ddots  & V_{k,n-1}^{k}(t) \\
V_{k,n}^{1}(t) & V_{k,n}^{2}(t) & \ldots  & V_{k,n}^{k}(t)%
\end{array}%
\right] .
\end{eqnarray*}%
After use matrix multiplication we get%
\begin{equation*}
R_{k,n}^{i}(t)=(-t_{k-1})V_{k,n-k+1}^{i}(t)+\ldots
+(-t_{2}(k-2))V_{k,n-1}^{i}(t)+kV_{k,n}^{i}(t).
\end{equation*}
\end{proof}

\bigskip

\begin{exam}
\bigskip We obtain $R_{4,5}^{3}(t)$ using Theorem (2.12)%
\begin{eqnarray*}
R_{4,5}^{3}(t)
&=&(-t_{3})V_{4,5-4+1}^{3}(t)+(-t_{2}(k-2))V_{4,5-3+1}^{3}(t)+kV_{4,5}^{3}(t)
\\
&=&(-t_{3})V_{4,2}^{3}(t)+(-t_{2}(4-2))V_{4,3}^{3}(t)+kV_{4,5}^{3}(t) \\
&=&(-t_{3})t_{3}+(-2t_{2})(t_{4}+t_{2}^{2})+4(t_{3}^{2}+t_{2}^{3})=6t_{2}t_{4}+2t_{2}^{3}+3t_{3}^{2}.
\end{eqnarray*}
\end{exam}
\bigskip

\begin{thrm}
\label{t1}Let $V_{k,n}^{i}(t)$ be the $k$ sequences of generalized
Van der Laan
polynomials, then%
\begin{equation*}
V_{k,n+m}^{i}(t)=\sum\limits_{j=1}^{k}V_{k,m}^{j}(t)V_{k,n-k+j}^{i}(t).
\end{equation*}
\bigskip
\end{thrm}

\bigskip

\begin{proof}
\bigskip We know that $\widetilde{V_{n}}=(V_{(k)})^{n}$. We may rewrite it as%
\begin{eqnarray*}
(V_{(k)})^{n+1} &=&(V_{(k)})^{n}(V_{(k)})=(V_{(k)})(V_{(k)})^{n} \\
&\Rightarrow &\widetilde{V_{n+1}}=\widetilde{V_{n}}\widetilde{V_{1}}=%
\widetilde{V_{1}}\widetilde{V_{n}}
\end{eqnarray*}%
and inductively%
\begin{equation}
\widetilde{V_{n+m}}=\widetilde{V_{n}}\widetilde{V_{m}}=\widetilde{V_{m}}%
\widetilde{V_{n}}.  \label{a2}
\end{equation}%
Consequently, any element of $\widetilde{V_{n+m}}$ is the product of
a row
of $\widetilde{V_{n}}$ and a column of $\widetilde{V_{m}};$ that is%
\begin{equation*}
V_{k,n+m}^{i}(t)=\sum\limits_{j=1}^{k}V_{k,m}^{j}(t)V_{k,n-k+j}^{i}(t).
\end{equation*}
\end{proof}

\bigskip

\begin{cor}
\bigskip \bigskip Taking $n=m$ in (\ref{a2}) we obtain $(\widetilde{V_{n}}%
)^{2}=\widetilde{V_{n}}\widetilde{V_{n}}=\widetilde{V_{n+n}}=\widetilde{%
V_{2n}}$\bigskip .
\end{cor}

\begin{thrm}
\bigskip $R_{k,n}^{i}(t)$ and $V_{k,n}^{i}(t)$ are $k$ sequences of
generalized Perrin and Van der Laan polynomials respectively, then%
\begin{equation*}
R_{k,n}^{i}(t)=kt_{k}V_{k,n-k}^{i}(t)+\cdots
+3t_{3}V_{k,n-3}^{i}(t)+2t_{2}V_{k,n-2}^{i}(t).
\end{equation*}
\end{thrm}

\bigskip

\begin{proof}
\bigskip
\begin{eqnarray*}
\widetilde{R_{n}} &=&R_{(k)}\widetilde{V_{n}}=R_{(k)}\widetilde{V_{1}}%
\widetilde{V_{n-1}} \\
&\Rightarrow &\left[
\begin{array}{cccc}
R_{k,n-k+1}^{1}(t) & R_{k,n-k+1}^{2}(t) & \ldots  & R_{k,n-k+1}^{k}(t) \\
\vdots  & \vdots  & \ldots  & \vdots  \\
R_{k,n-1}^{1}(t) & R_{k,n-1}^{2}(t) & \ddots  & R_{k,n-1}^{k}(t) \\
R_{k,n}^{1}(t) & R_{k,n}^{2}(t) & \ldots  & R_{k,n}^{k}(t)%
\end{array}%
\right]  \\
&=&\left[
\begin{array}{c}
(-t_{k-1},\ldots ,-t_{2}(k-2),0,k).(V_{(k)})^{-(k)} \\
(-t_{k-1},\ldots ,-t_{2}(k-2),0,k).(V_{(k)})^{-(k-1)} \\
\vdots  \\
(-t_{k-1},\ldots ,-t_{2}(k-2),0,k).(V_{(k)})^{-1} \\
(-t_{k-1},\ldots ,-t_{2}(k-2),0,k)%
\end{array}%
\right] \left[
\begin{array}{ccccc}
0 & 1 & 0 & \ldots  & 0 \\
0 & 0 & 1 & \ldots  & 0 \\
\vdots  & \vdots  & \vdots  & \ddots  & \vdots  \\
0 & 0 & 0 & \ldots  & 1 \\
t_{k} & t_{k-1} & t_{k-2} & \ldots  & 0%
\end{array}%
\right]  \\
&&\left[
\begin{array}{cccc}
V_{k,n-k}^{1}(t) & V_{k,n-k}^{2}(t) & \ldots  & V_{k,n-k}^{k}(t) \\
\vdots  & \vdots  & \ldots  & \vdots  \\
V_{k,n-2}^{1}(t) & V_{k,n-2}^{2}(t) & \ddots  & V_{k,n-2}^{k}(t) \\
V_{k,n-1}^{1}(t) & V_{k,n-1}^{2}(t) & \ldots  & V_{k,n-1}^{k}(t)%
\end{array}%
\right]  \\
&=&\left[
\begin{array}{c}
(-t_{k-1},\ldots ,-t_{2}(k-2),0,k).(V_{(k)})^{-(k-1)} \\
(-t_{k-1},\ldots ,-t_{2}(k-2),0,k).(V_{(k)})^{-(k-2)} \\
\vdots  \\
(-t_{k-1},\ldots ,-t_{2}(k-2),0,k) \\
(kt_{k},\ldots ,3t_{3},2t_{2},0)%
\end{array}%
\right]  \\
&&\left[
\begin{array}{cccc}
V_{k,n-k}^{1}(t) & V_{k,n-k}^{2}(t) & \ldots  & V_{k,n-k}^{k}(t) \\
\vdots  & \vdots  & \ldots  & \vdots  \\
V_{k,n-2}^{1}(t) & V_{k,n-2}^{2}(t) & \ddots  & V_{k,n-2}^{k}(t) \\
V_{k,n-1}^{1}(t) & V_{k,n-1}^{2}(t) & \ldots  & V_{k,n-1}^{k}(t)%
\end{array}%
\right] .
\end{eqnarray*}%
Using matrix multiplication we obtain
\begin{equation*}
R_{k,n}^{i}(t)=kt_{k}V_{k,n-k}^{i}(t)+\cdots
+3t_{3}V_{k,n-3}^{i}(t)+2t_{2}V_{k,n-2}^{i}(t).
\end{equation*}
\end{proof}

\bigskip

\section{Generalized order-$k$ Van der Laan and Perrin numbers}

\bigskip

In this section we define generalized order-$k$ Van der Laan numbers
$v_{k,n}
$ and $k$ sequences of the generalized order-$k$ Van der Laan numbers $%
v_{k,n}^{i}$ by the help of $k$ sequences of Generalized Van der
Laan
Polynomials. In addition we define generalized order-$k$ Perrin numbers $%
r_{k,n}$ and $k$ sequences of the generalized order-$k$ Perrin numbers $%
r_{k,n}^{i}$ by the help of $k$ sequences of Generalized Perrin
Polynomials.

\bigskip

\begin{definition}
For $t_{s}=1,$ $2\leq s\leq k$, the generalized Van der Laan polynomial $%
V_{k,n}(t)$ and $V_{(k)}^{\infty }$ together are reduced to%
\begin{equation*}
v_{k,n}=\sum\limits_{j=2}^{k}v_{k,n-j}
\end{equation*}%
with boundary conditions%
\begin{equation*}
v_{k,1-k}=v_{k,2-k}=\ldots =v_{k,-2}=0,\text{ }v_{k,-1}=1\ \text{and }%
v_{k,0}=0,
\end{equation*}%
which is called generalized order-$k$ Van der Laan numbers(GO$k$V).
\end{definition}

When $k=3$, it is reduced to ordinary Van der Laan numbers.

\bigskip

\begin{definition}
For $t_{s}=1$, $2\leq s\leq k,$ $V_{k,n}^{i}(t)$ can be written explicitly as%
\begin{equation}
v_{k,n}^{i}=\sum\limits_{j=2}^{k}v_{k,n-j}^{i} \label{vtn}
\end{equation}
for $n>0$ and $1\leq i\leq k,$\ with boundary conditions%
\begin{equation*}
v_{k,n}^{i}=\left\{
\begin{array}{lc}
1\ \ \ \ \ \ \ \ \text{ \ if }i-n=k, &  \\
0\ \text{ \ \ \ \ \ \ \ \ otherwise} &
\end{array}%
\right. \ \
\end{equation*}%
for $1-k\leq n\leq 0,$ where $v_{k,n}^{i}$ is the $n$-th term of
$i$-th sequence. This generalization is called $k$ sequences of the
generalized order-$k$ Van der Laan numbers($k$SO$k$V).
\end{definition}

\bigskip When $i=k=3,$ we obtain ordinary Van der Laan numbers and for any
integer $k,$ $v_{k,n-1}^{k}=v_{k,n}$.

\begin{exam}
Substituting $k=3$ and $i=2$ we obtain the generalized order-$3$ Van
der
Laan sequence as;%
\begin{equation*}
v_{3,-2}^{2}=0,\text{ }v_{3,-1}^{2}=1,\text{ }v_{3,0}^{2}=0,\text{ }%
v_{3,1}^{2}=1,\text{ }v_{3,2}^{2}=1,\text{ }v_{3,3}^{2}=1,\text{ }%
v_{3,4}^{2}=2,\ldots
\end{equation*}
\end{exam}

\bigskip

\bigskip We firtly give some properties of $k$ sequences of the generalized
order-$k$ Van der Laan numbers($k$SO$k$V) using properties of $k$
sequences of generalized Van der Laan polynomials($V_{k,n}^{i}(t)$).

\bigskip

\begin{cor}
Matrix multiplication and (\ref{vtn}) can be used to obtain
\begin{equation*}
V_{n}^{\sim }=A_{1}^{n}
\end{equation*}%
where\
\begin{equation}
A_{1}=\left[
\begin{array}{ccccccc}
0 & 1 & 0 & 0 & \ldots  & 0 & 0 \\
0 & 0 & 1 & 0 & \ldots  & 0 & 0 \\
0 & 0 & 0 & 1 & \cdots  & 0 & 0 \\
\vdots  & \vdots  & \vdots  &  & \ddots  & \vdots  & \vdots  \\
0 & 0 & 0 & 0 & \ldots  & 0 & 1 \\
1 & 1 & 1 & 1 & \ldots  & 1 & 0%
\end{array}%
\right] _{k\times k}=\left[
\begin{array}{cccc}
0 &  &  &  \\
0 &  & I &  \\
0 &  &  &  \\
1 & \ldots  & 1 & 0%
\end{array}%
\right] _{k\times k}\ \ \ \ \ \ \
\end{equation}%
where $I$ is ($k-1)\times (k-1)$ identity matrix and $V_{n}^{\sim }$
is a
matrix as;%
\begin{equation}
V_{n}^{\sim }=\left[
\begin{array}{cccc}
v_{k,n-k+1}^{1} & v_{k,n-k+1}^{2} & \ldots  & v_{k,n-k+1}^{k} \\
\vdots  & \vdots  & \ldots  & \vdots  \\
v_{k,n-1}^{1} & v_{k,n-1}^{2} & \ddots  & v_{k,n-1}^{k} \\
v_{k,n}^{1} & v_{k,n}^{2} & \ldots  & v_{k,n}^{k}%
\end{array}%
\right]   \label{vn}
\end{equation}%
which is contained by $k\times k$ block of\ $V_{(k)}^{\infty }$ for $t_{i}=1$%
, $2\leq i\leq k.$\ \ \ \ \ \ \ \ \ \ \ \ \ \ \ \ \ \ \ \ \ \ \ \ \
\ \ \ \ \ \ \ \ \ \ \ \ \ \ \ \ \ \ \ \ \ \ \ \ \ \ \ \ \ \ \ \ \ \
\ \ \ \ \ \ \ \ \ \ \ \ \ \ \ \ \ \ \ \ \ \ \ \ \ \ \ \ \ \ \ \ \ \
\ \ \ \ \ \ \ \ \ \
\end{cor}

\begin{proof}
It is clear that $V_{1}^{\sim }=A_{1}$ and $V_{n+1}^{\sim
}=A_{1}V_{n}^{\sim }$ by (\ref{vtn})$.$ So by induction, we have
$V_{n}^{\sim }=A_{1}^{n}$.
\end{proof}

\bigskip

\begin{cor}
\bigskip Let $V_{n}^{\sim }$ be as in (\ref{vn}). Then
\begin{equation*}
\det V_{n}^{\sim }=\left\{
\begin{array}{c}
1\text{ \ \ \ \ \ \ if k is odd,} \\
(-1)^{n}\text{ \ if k is even.}%
\end{array}%
\right.
\end{equation*}
\end{cor}

\bigskip

\begin{proof}
\bigskip Obvious from (\ref{a1}).
\end{proof}

\bigskip

\begin{cor}
\bigskip Let $v_{k,n}^{i}$ be the $i$-th sequences of k sequences of
generalized order-$k$ Van der Laan Numbers, for $1\leq i\leq k.$
Then, for
all positive integers $n$ and $m$%
\begin{equation*}
v_{k,n+m}^{i}=\sum\limits_{j=1}^{k}v_{k,m}^{j}v_{k,n-k+j}^{i}.
\end{equation*}%
\bigskip
\end{cor}

\bigskip

\begin{proof}
Obvious from Theorem (\ref{t1}).
\end{proof}

\bigskip

\bigskip

\begin{cor}
\bigskip Let$\ v_{k,n}^{\text{ }i}$ be the $i$-th sequences of $k$SO$k$V. Then, for
$n>1-k$,%
\begin{equation}
v_{k,n}^{1}=v_{k,n-1}^{\text{ }k}=v_{k,n-2}^{\text{ }k-1}.
\end{equation}
\end{cor}

\begin{proof}
It is obvious from Definition (3.2) that this sequence are equal
with index iteration.
\end{proof}

\bigskip

\begin{lem}
Let$\ v_{k,n}^{\text{ }i}$ be the $i$-th sequences of $k$SO$k$V then%
\begin{equation}
v_{k,n}^{\text{ }i}=v_{k,n}^{\text{ }i-1}+\text{ }v_{k,n-i}^{\text{
}k} \label{a4}
\end{equation}%
for $n>1-k+i$ and $1<i\leq k.$
\end{lem}

\bigskip

\begin{proof}
\bigskip \bigskip Assume for\ $n>1-k+i,$ $v_{k,n}^{\text{ }i}-v_{k,n}^{\text{
}i-1}=t_{n}$ and show $t_{n}=$ $v_{k,n-i}^{\text{ }k}.$

First we obtain initial conditions for $t_{n}$ by using initial
conditions
of \ $i$-th and $(i-1)$-th sequences of $k$SO$k$V simultaneously as follows;%
\begin{equation*}
\begin{tabular}{|c|c|c|c|}
\hline
$n\setminus i$ & $v_{k,n}^{\text{ }i}$ & $v_{k,n}^{\text{ }i-1}$ & $%
t_{n}=v_{k,n}^{\text{ }i}-v_{k,n}^{\text{ }i-1}$ \\ \hline $1-k$ &
$0$ & $0$ & $0$ \\ \hline $2-k$ & $0$ & $0$ & $0$ \\ \hline $\vdots
$ & $\vdots $ & $\vdots $ & $\vdots $ \\ \hline $i-k-2$ & $0$ & $0$
& $0$ \\ \hline $i-k-1$ & $0$ & $1$ & $-1$ \\ \hline $i-k$ & $1$ &
$0$ & $1$ \\ \hline $i-k+1$ & $0$ & $0$ & $0$ \\ \hline $\vdots $ &
$\vdots $ & $\vdots $ & $\vdots $ \\ \hline
$0$ & $0$ & $0$ & $0$%
\end{tabular}%
\end{equation*}%
Since initial conditions of $t_{n}$ are equal to initial condition of $%
v_{k,n}^{\text{ }k}$ with index iteration. Then we have%
\begin{equation*}
t_{n}=v_{k,n-i}^{\text{ }k}.
\end{equation*}
\end{proof}

\bigskip

\bigskip We give the following Theorem by using generalization of MacHenry
in [10].

\begin{thrm}
\bigskip Let$\ v_{k,n}^{\text{ }i}$ be the $i$-th sequences of $k$SO$k$V$,$
then for $n\geq 1$ and $1\leq i\leq k$%
\begin{equation*}
v_{k,n}^{\text{ }i}=v_{k,n-1}^{\text{ }k}+v_{k,n-2}^{\text{
}k}+\cdots +v_{k,n-i}^{\text{
}k}=\sum\limits_{m=1}^{i}v_{k,n-m}^{k}.
\end{equation*}
\end{thrm}

\bigskip

\begin{proof}
Writting equalty (\ref{a4}) recursively we have
\begin{eqnarray*}
v_{k,n}^{i+1}-v_{k,n}^{i} &=&v_{k,n-i-1}^{k} \\
v_{k,n}^{i+2}-v_{k,n}^{i+1} &=&v_{k,n-i-2}^{k} \\
&&\vdots  \\
v_{k,n}^{k-2}-v_{k,n}^{k-3} &=&v_{k,n-k+2}^{k} \\
v_{k,n}^{k-1}-v_{k,n}^{k-2} &=&v_{k,n-k+1}^{k}
\end{eqnarray*}%
and adding these equations side by side we obtain
\begin{equation*}
v_{k,n}^{k-1}-v_{k,n}^{i}=v_{k,n-k+1}^{k}+v_{k,n-k+2}^{k}+\cdots
+v_{k,n-i-2}^{k}+v_{k,n-i-1}^{k}.
\end{equation*}%
And using the equation $v_{k,n}^{k-1}=v_{k,n+1}^{k}$ and Definition
(3.4) we obtain
\begin{eqnarray*}
v_{k,n}^{i} &=&v_{k,n+1}^{k}-(v_{k,n-k+1}^{k}+v_{k,n-k+2}^{k}+\cdots
+v_{k,n-i-2}^{k}+v_{k,n-i-1}^{k}) \\
&=&v_{k,n-1}^{k}+v_{k,n-2}^{k}+\cdots +v_{k,n-k+1}^{k} \\
&&-(v_{k,n-k+1}^{k}+v_{k,n-k+2}^{k}+\cdots
+v_{k,n-i-2}^{k}+v_{k,n-i-1}^{k})
\\
&=&v_{k,n-1}^{k}+v_{k,n-2}^{k}+\cdots +v_{k,n-i}^{k}.
\end{eqnarray*}
\end{proof}

\bigskip \bigskip

\bigskip Now we initiate to the generalized Perrin numbers.

\begin{definition}
\bigskip For $t_{s}=1,$ $2\leq s\leq k$, the generalized Perrin polynomial $%
R_{k,n}(t)$ and matrix $R_{(k)}^{\infty }$ together are reduced to%
\begin{equation}
r_{k,n}=\sum\limits_{j=2}^{k}r_{k,n-j}
\end{equation}%
with boundary conditions%
\begin{equation*}
r_{k,1-k}=(k-2),\text{ }r_{k,2-k}=\ldots =r_{k,-2}=\text{
}r_{k,-1}=-1\ \text{and }r_{k,0}=k,
\end{equation*}%
which is called generalized order-$k$ Perrin numbers(GO$k$R).
\end{definition}

When $k=3$, it is reduced to ordinary Perrin numbers; $%
(1,(-1),3,0,2,3,2,5,5,7,\ldots )$ with iterating index by two.

\bigskip

We rewrite matrix (\ref{rk}) for $t_{s}=1$, $2\leq s\leq k$ we obtain%
\begin{equation*}
R_{(k1)}=[a_{n,i}]_{k\times k}=\left[
\begin{array}{c}
((-1),\text{ }(-2),\ldots ,(k-2),0,k).(A_{1})^{-(k-1)} \\
((-1),\text{ }(-2),\ldots ,(k-2),0,k).(A_{1})^{-(k-2)} \\
\vdots  \\
((-1),\text{ }(-2),\ldots ,(k-2),0,k).(A_{1})^{-1} \\
((-1),\text{ }(-2),\ldots ,(k-2),0,k)%
\end{array}%
\right] .
\end{equation*}

\begin{definition}
For $t_{s}=1$, $2\leq s\leq k,$ $R_{k,n}^{i}(t)$ can be written explicitly as%
\begin{equation*}
r_{k,n}^{i}=\sum\limits_{j=2}^{k}r_{k,n-j}^{i}
\end{equation*}%
for $n>0$ and $1\leq i\leq k,$\ with boundary conditions%
\begin{equation*}
r_{k,n}^{i}=[a_{k+n,i}]_{k\times k}=R_{(k1)}
\end{equation*}%
for $1-k\leq n\leq 0,$ where $r_{k,n}^{i}$ is the $n$-th term of
$i$-th sequence. This generalization is called $k$ sequences of the
generalized order-$k$ Perrin numbers($k$SO$k$R).
\end{definition}

\bigskip When $i=k=3,$ we obtain ordinary Perrin numbers and for any
integer $k\geq 3,$ $r_{k,n}^{k}=r_{k,n}$.

\begin{cor}
Let $r_{k,n}^{i}$ and $v_{k,n}^{i}$ be $k$ sequences of generalized
Perrin and Van der Laan numbers respectively, then
\begin{equation*}
r_{k,n}^{i}=kv_{k,n}^{i}-(v_{k,n-k+1}^{i}+\ldots
+(k-2)v_{k,n-2}^{i}).
\end{equation*}
\end{cor}

\bigskip

\begin{cor}
Let $r_{k,n}^{i}$ and $v_{k,n}^{i}$ be $k$ sequences of generalized
Perrin
and Van der Laan numbers respectively, then%
\begin{equation*}
r_{k,n}^{i}(t)=kv_{k,n-k}^{i}+\cdots +3v_{k,n-3}^{i}+2v_{k,n-2}^{i}.
\end{equation*}
\end{cor}

\bigskip

\begin{conclusion}
\bigskip There are a lot of studies on Fibonacci and Lucas numbers and on their
generalizations. In this paper we showed that these studies can be
transferred to the Van der Laan and Perrin numbers. Since our
definition of these number are polynomial based, it has great amount
of application area.
\end{conclusion}

\bigskip

\bigskip

\end{document}